\documentclass[10pt,letterpaper]{article}
\usepackage[utf8]{inputenc}
\usepackage{amsmath}
\usepackage{amsfonts}
\usepackage{amssymb}
\usepackage{graphicx}
\usepackage{verbatim}
\usepackage{mathrsfs}
\usepackage{upref,amsthm,amsxtra,exscale}
\usepackage{cite}
\usepackage{dsfont}
\usepackage[colorlinks=true,urlcolor=blue,
citecolor=red,linkcolor=blue,linktocpage,pdfpagelabels,
bookmarksnumbered,bookmarksopen]{hyperref}

\usepackage{fullpage}

\newtheorem{theorem}{Theorem}[section]
\newtheorem{corollary}[theorem]{Corollary}

\newtheorem{lemma}[theorem]{Lemma}

\newtheorem{definition}[theorem]{Definition}
\newtheorem{examples}[theorem]{Examples}

\numberwithin{equation}{section}

\def\R{\mathbb{R}}
\def\r{\mathbb{R}}
\def\rn{\mathbb{R}^N}
\def\z{\mathbb{Z}}
\def\z2{\mathbb{Z}_2}

\def\s1{\mathbb{S}^1}
\def\n{\mathbb{N}}
\def\cc{\mathbb{C}}
\def\eps{\varepsilon}
\def\rh{\rightharpoonup}
\def\io{\int_{\Omega}}
\def\irn{\int_{\r^N}}
\def\vp{\varphi}

\def\vr{\varrho}
\def\o{\Omega}

\def\cA{\mathcal{A}}
\def\cB{\mathcal{B}}
\def\cC{\mathcal{C}}

\def\cE{\mathcal{E}}

\def\cN{\mathcal{N}}
\def\cO{\mathcal{O}}
\def\cP{\mathcal{P}}

\def\cU{\mathcal{U}}

\def\cW{\mathcal{W}}
\def\cY{\mathcal{Y}}
\def\cZ{\mathcal{Z}}

\def\what{\widehat}

\def\d{\,\mathrm{d}}

\def\dist{\mathrm{dist}}
\def\dim{\mathrm{dim}}
\def\gen{\mathrm{genus}}

\author{Mónica Clapp  \ and \ Carlos Culebro}

\title{Multiple nodal solutions to a scalar field equation with double-power nonlinearity and zero mass at infinity}
\date{}

\begin{document}
\maketitle

\begin{abstract}
We consider the nonlinear elliptic equation
\begin{equation*}
-\Delta u + V(x)u = f(u), \qquad u\in D^{1,2}_0(\o),
\end{equation*}
in an exterior domain $\o$ of $\rn$, where $V$ is a scalar potential that decays to zero at infinity and the nonlinearity $f$ is subcritical at infinity and supercritical near the origin. 
Under weak symmetry assumptions, we provide conditions that guarantee that this problem has a prescribed number of sign-changing solutions. In particular, we show that in dimensions $N\geq 4$ there are numerous examples of exterior domains with finite symmetries in which the problem has a predetermined number of nodal solutions.
\smallskip

\smallskip

\end{abstract}
	
\section{Introduction}
	
This paper is concerned with the existence of multiple sign-changing solutions to the problem
\begin{equation}
\label{eq:problem}
-\Delta u + V(x)u = f(u), \qquad u\in D^{1,2}_0(\o),
\end{equation}
in a domain $\o$ of $\rn$, $N\geq 3$, where  $V$ is a scalar potential that decays to zero at infinity and the nonlinearity $f$ is subcritical at infinity and supercritical near the origin. The space $D^{1,2}_0(\o)$ is the closure of $\cC_c^\infty(\o)$ in the Sobolev space $D^{1,2}(\rn):=\{u\in L^{2^*}(\rn):\nabla u\in L^2(\rn,\rn)\}$, equipped with its usual norm
\begin{equation} \label{eq:usual norm}
\|u\|:=\Big(\irn|\nabla u|^2\Big)^{1/2},
\end{equation}
and $2^*:=\frac{2N}{N-2}$ is the critical Sobolev exponent.

Such equations arise, for instance, in some particle physics problems related to the non-Abelian gauge theory underlying strong interaction, called quantum chromodynamics or QCD. Their solutions lead to some special solutions of the pure Yang-Mills equations via the 't Hooft Ansatz; see \cite{g}.

In their seminal paper \cite{bl} Berestycki and Lions showed that, when $\o=\rn$ and $V=0$, this problem has a ground state solution which is positive, radially symmetric and decreasing in the radial direction. 

For $\o=\rn$ Badiale and Rolando established the existence of a positive radial solution to \eqref{eq:problem} when $V$ is radial \cite{br}. Without assuming any symmetries on $V$, Benci, Grisanti and Micheletti showed in \cite{bgm1} that, under suitable hypotheses, \eqref{eq:problem} has a positive ground state if $V\leq 0$ and that it has no ground state if $V\geq0$ and $V\neq 0$. The existence of a positive bound state was established in \cite{cm} whenever $V$ decays to zero at a suitable rate and the limit problem (i.e., \eqref{eq:problem} with $V\equiv 0$) has a unique positive solution.

In an exterior domain $\o$ the existence of a positive solution to the problem \eqref{eq:problem} with $V=0$ was first established by Benci and Micheletti in \cite{bm} for domains whose complement has a sufficiently small diameter, and was then extended to general exterior domains by Khatib and Maia in \cite{km}. The existence of multiple positive solutions in the complement of a ball, the number of which increases as the radius of the ball increases, was shown in \cite{cmp1}.

For $\o=\rn$ and $V=0$ there are also some results on sign-changing solutions. Taking advantage of the symmetries that \eqref{eq:problem} presents in this case, Mederski established the existence of a non-radial sign-changing solution if $N\geq 4$ and of infinitely many such solutions if, in addition, $N\neq 5$; see \cite{m1,m2}. This last result is based on the fact that in such dimensions there exist groups $G$ of linear isometries that admit an involution and have the property that the $G$-orbit of each point $x\neq 0$ in $\rn$ has infinite cardinality. As shown by Bartsch and Willem in \cite{bw}, this type of symmetries produce a sign change by construction.

A delicate aspect when approaching problem \eqref{eq:problem} using variational methods is the lack of compactness of the energy functional. Considering symmetries with infinite $G$-orbits restores compactness and allows applying standard variational methods to obtain infinitely many solutions. On the other hand, when any group action on the domain has finite $G$-orbits, the lack of compactness prevails and such methods cannot be applied. To our knowledge, there are no results on multiple solutions in such case.

The main objective of this work is to establish the existence of \emph{multiple sign-changing solutions} in exterior domains that are invariant under a group $G$ of linear isometries and \emph{have finite $G$-orbits}. Recall that $\o$ is an exterior domain if its complement is bounded (possibly empty). 

We make the following assumptions about the potential and the nonlinearity. 

\begin{itemize}
\item[$(V_1)$] $V\in L^{N/2}(\rn)\cap L^r(\rn)$ for some $r > N/2$ and $\irn|V^-|^{N/2} < S^{N/2}$, where $V^-:=\min\{0,V\}$ and $S$ is the best constant for the Sobolev embedding $D^{1,2}(\rn)\hookrightarrow L^{2^*}(\rn)$.
\item [$(f_1)$] $f\in\cC^1(\r)$ and there are $A_1>0$ and $2<p<2^*<q$ such that, for $m=-1,0,1,$
\begin{equation}
|f^{(m)}(s)| \leq 
\begin{cases}
A_1|s|^{p-(m+1)} &\text{if \ } |s|\geq 1, \\
A_1|s|^{q-(m+1)} &\text{if \ } |s| \leq 1, 
\end{cases}
\end{equation}
where $f^{(-1)}:=F$, $f^(0):=f $, $f^{(1)}:= f'$ and $F(s):= \int_{0}^{s} f(t)\d t$.
\item [$(f_2)$] There exists $\theta >2$ such that $0<\theta F(s)\leq f(s)s<f'(s)s^2$ for all $s\neq 0$. 
\item [$(f_3)$] $f$ is odd, i.e., $f(-s)=-f(s)$ for all $s\in\r$.
\end{itemize}

Let $G$ be a closed subgroup of the group $O(N)$ of linear isometries of $\rn$. Recall that the $G$-orbit of a point $x\in\rn$ is the set
$$Gx:=\{gx:g\in G\}.$$
We write $\#Gx$ for its cardinality. A domain $\o$ is said to be $G$-invariant if $Gx\subset\o$ for every $x\in\o$ and a function $u:\o\to\r$ is $G$-invariant if it is constant on each $G$-orbit of $\o$.

Our first result regards bounded domains, possibly without symmetries, and exterior domains having symmetries with infinite orbits.

\begin{theorem} \label{thm:main1}
Let $V$ and $f$ satisfy $(V_1), (f_1), (f_2)$ and $(f_3)$. Assume that  $\o$ and $V$ are $G$-invariant, where $G$ is a closed subgroup of $O(N)$. If, either
\begin{itemize}
\item $\o$ is bounded, or
\item $\o$ is an exterior domain and $\#Gx = \infty$ for every $x\in\rn\smallsetminus\{0\}$,
\end{itemize}
the problem \eqref{eq:problem} has one positive and infinitely many sign-changing $G$-invariant solutions. The positive one has least energy among all nontrivial $G$-invariant solutions.
\end{theorem}

\begin{examples}
The following groups satisfy $\#Gx = \infty$ for every $x\in\rn\smallsetminus\{0\}$.
\begin{itemize}
\item[$(a)$] If $G=O(N)$, then $Gx= S_{|x|}^{N-1}:=\{y\in\rn:|y|=|x|\}$, the sphere of radius $|x|$ in $\rn$ centered at the origin.
\item[$(b)$] If $N=n_1+\cdots+n_m$ with $n_i\geq 2$ for all $i$, and $G=O(n_1)\times\cdots\times O(n_m)$, then 
$$Gx= S^{n_1-1}_{|x_1|}\times\cdots\times S^{n_m-1}_{|x_m|}\qquad\text{for every \ }x=(x_1,\ldots,x_m)\in\r^{n_1}\times\cdots\times\r^{n_m}\equiv\rn.$$
\end{itemize}
\end{examples}

These examples apply, in particular, when $\o=\rn$.

Our next result applies to groups that have finite orbits. It ensures the existence of a prescribed number of solutions if the cardinality of the orbits of the domain is sufficiently large. 

We write $B_R:=\{x\in\rn:|x|<R\}$.

\begin{theorem} \label{thm:main2}
Assume that $V$ is radial and that $V$ and $f$ satisfy $(V_1),(f_1), (f_2)$ and $(f_3)$. Given $R>0$, there exists an increasing sequence $(\ell_m)$ of positive real numbers, depending only on $R$, with the following property: If $\rn \smallsetminus\o\subset B_R$ and there exists a closed subgroup $G$ of $O(N)$ such that $\o$ is $G$-invariant and
$$\min_{x\in\rn\smallsetminus\{0\}}\# Gx > \ell_m,$$
then the problem \eqref{eq:problem} has at least $m$ pairs $\pm u_1,\ldots,\pm u_m$ of $G$-invariant solutions such that $u_1$ is positive and has least energy among all nontrivial $G$-invariant solutions, and $u_2,\ldots,u_m$ change sign. Their energy satisfies
$$\frac{1}{2}\io (|\nabla u_j|^2+Vu_j^2)-\io F(u_j)\leq\ell_jc_\infty,\qquad\text{for each \ }j=1,\ldots,m,$$
where $c_\infty$ is the ground state energy of the limit problem 
\begin{equation}
\label{eq:limit problem}
-\Delta u = f(u), \qquad u\in D^{1,2}(\rn).
\end{equation}
\end{theorem}

\begin{examples}
The following two examples show that, if $N\geq 4$, there are many closed subgroups $G$ of $O(N)$ satisfying $\min_{x\in\rn\smallsetminus\{0\}}\# Gx > \ell$ for any given $\ell$.
\begin{itemize}
\item[$(a)$]If $N$ is even then, for any given $n\in\n$, the cyclic group $\mathbb{Z}_n:=\{\mathrm{e}^{2\pi\mathrm{i}j/n}:j=0,\ldots,n-1\},$ acting on $\rn$ as
$$\mathrm{e}^{2\pi\mathrm{i}j/n}(z_1,\ldots,z_{N/2})=(\mathrm{e}^{2\pi\mathrm{i}j/n}z_1,\ldots,\mathrm{e}^{2\pi\mathrm{i}j/n}z_{N/2}),\qquad(z_1,\ldots,z_{N/2})\in\cc^{N/2}\equiv\rn,$$
has the property that $\#\mathbb{Z}_nx=n$ for every $x\in\rn\smallsetminus\{0\}$.
\item[$(b)$]If $N\geq 4$ then, for any given $n\in\n$, there are closed subgroups $G$ of $O(N)$ such that
$$\min_{x\in\rn\smallsetminus\{0\}}\# Gx = n.$$
For instance, the group $G_n:=\mathbb{Z}_n\times O(N-2)$ acting on $\rn$ as
\begin{align*}
(\mathrm{e}^{2\pi\mathrm{i}j/n},g)(z,y):=(\mathrm{e}^{2\pi\mathrm{i}j/n}z,gy)\qquad\text{for all \ }(z,y)\in \cc\times\r^{N-2}\equiv\rn, \ g\in O(N-2),
\end{align*}
has this property.
\end{itemize}
In contrast, $O(3)$ does not have subgroups with  finite orbits of arbitrarily large cardinality. The complete list of subgroups of $O(3)$ may be found in \emph{\cite[Section 8.2]{b}}.
\end{examples}
These examples, together with Theorem \ref{thm:main2}, show that for $N\geq 4$ there are many exterior domains that admit only symmetries having finite orbits in which the problem \eqref{eq:problem} has a prescribed number of solutions; see Figure \ref{fig}.
\begin{figure}[htbp]
        \centering
        \includegraphics[width=0.30\textwidth]{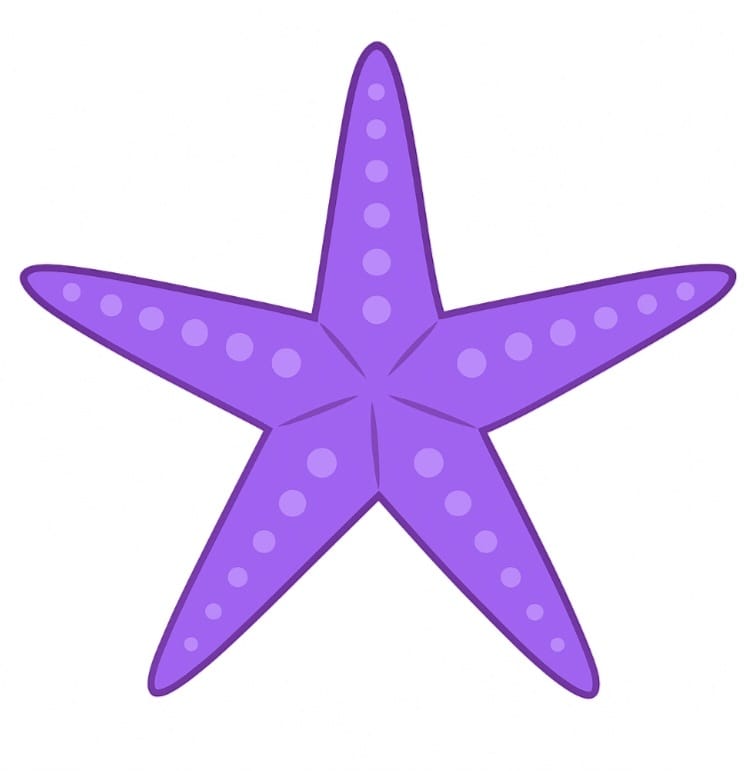}
        \qquad\qquad\qquad
        \includegraphics[width=0.30\textwidth]{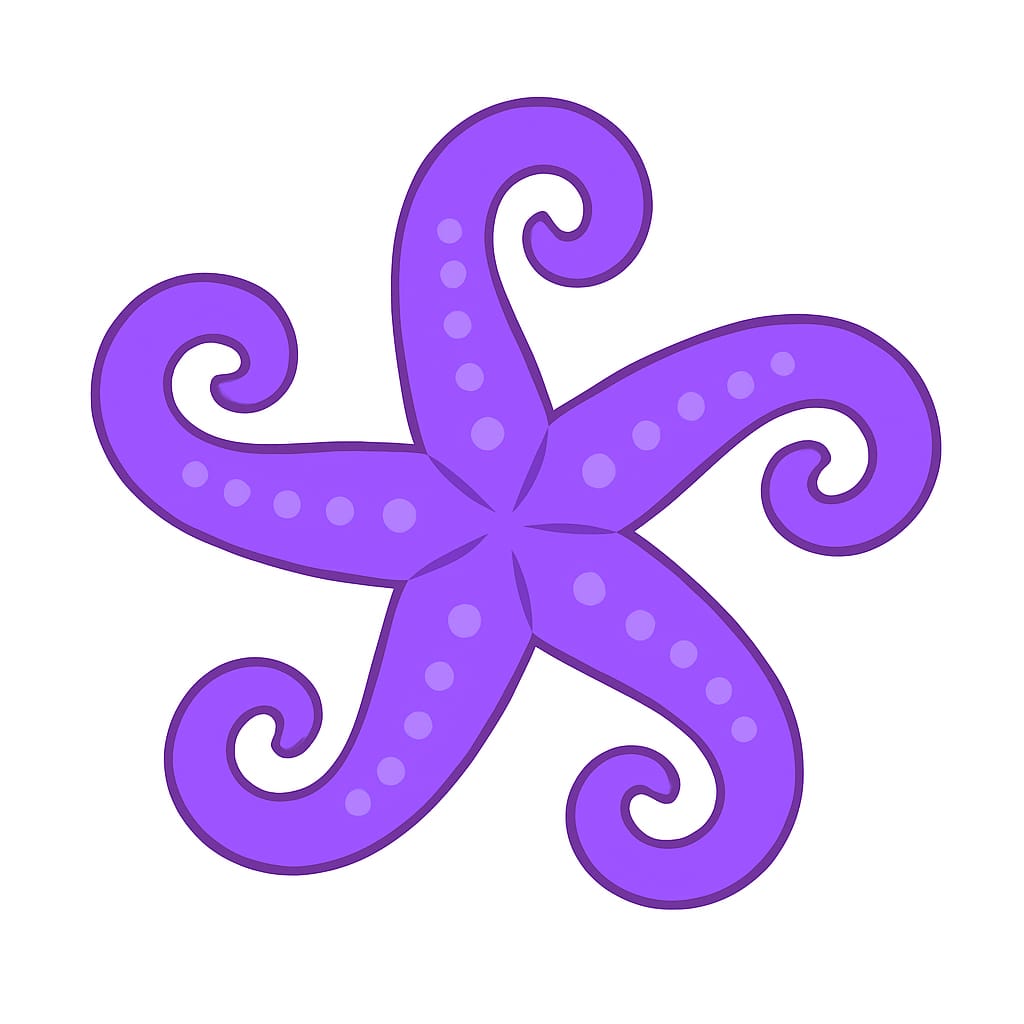}
        \caption{Projections onto $\cc$ of two bounded $[\mathbb{Z}_5\times O(N-2)]$-invariant subsets of $\cc\times\r^{N-2}$, $N\geq 4$. If $\ell_m<5$, the problem \eqref{eq:problem} has a positive and $m-1$ sign-changing solutions in the complement of each of these sets.}
        \label{fig}
  \end{figure}

Unlike the groups considered in \cite{bw,m1,m2}, we do not require that $G$ be provided with an involution. The solutions we obtain do not change sign by construction, but are given by a symmetric mountain pass theorem for nodal solutions; see Theorem \ref{thm:nodal principle}. This theorem is obtained by analyzing the negative gradient flow of the energy functional and uses a topological invariant that measures only those changes in the topology of sublevel sets that are produced by solutions that change sign. It is inspired by results in \cite{cp,cw}.

The Palais-Smale condition only holds up to a certain level, determined by the minimum cardinality of a $G$-orbit in $\o$; see Theorem \ref{thm:compactness}. To establish this level we use a concentration compactness argument whose main ingredient is a version of Lions' vanishing lemma proved in \cite{cm}.

Theorem \ref{thm:main1} follows immediately from Theorems \ref{thm:compactness} and \ref{thm:nodal principle}. The proof of Theorem \ref{thm:main2} requires additional work. The numbers $\ell_m$ are explicitly defined in terms of the infimum of all sums of the ground state energies in $m$ disjoint annuli contained in $\rn\smallsetminus B_R$; see \eqref{eq:ell}. This allows these numbers to be estimated. To prove Theorem \ref{thm:main2}, we introduce an $m$-dimensional linear subspace $W_m$ of the space of radial functions in $D^{1,2}_0(\rn\smallsetminus B_R)$ on which the maximum $d_m$ of the energy is less than $\ell_mc_\infty$. Thus, under the assumptions of Theorem \ref{thm:main2}, Theorem \ref{thm:compactness} guarantees that the Palais-Smale condition holds below $d_m$ and Theorem \ref{thm:nodal principle} yields the existence of one positive and $m-1$ sign-changing solutions.

It is worth noting that, for $N\geq 5$, an upper bound for the lowest possible energy of a sign-changing solution to \eqref{eq:problem} in $\rn$ with $V=0$ was recently given in \cite{cmp2}.

This paper is organized as follows. In Section 2 we study the symmetric variational problem and establish a level below which the Palais-Smale condition holds. In Section 3 we derive a symmetric mountain-pass theorem for nodal solutions. Section 4 is devoted to the proofs of Theorems \ref{thm:main1} and \ref{thm:main2}.

\section{A compactness criterion}

Let $G$ be a closed subgroup of $O(N)$. Throughout this section we assume that $\o$ and $V$ are $G$-invariant, and that assumptions $(V_1)$, $(f_1)$ and $(f_2)$ are satisfied. 

For $u,v\in D^{1,2}_0(\o)$ set
$$\langle u,v\rangle_V:=\io(\nabla u\cdot\nabla v+Vuv)\qquad\text{and}\qquad\|u\|_V^2:=\io(|\nabla u|^2+Vu^2).$$
Assumption $(V_1)$ and the Sobolev inequality imply that these expresions are well defined and that $\|\cdot\|_V$ is equivalent to the standard norm \eqref{eq:usual norm} of $D^{1,2}_0(\o)$. It follows from $(f_1)$ that $|F(s)|\leq A_1|s|^{2^*}$ and $|f(s)|\leq A_1|s|^{2^*-1}$. Hence, the functional $I_V:D^{1,2}_0(\o)\to\r$ given by
$$I_V(u):=\frac{1}{2}\|u\|_V^2-\io F(u)$$
is well defined. As shown in \cite[Lemma 2.6]{bm} and \cite[Proposition 3.8]{bpr}, $I_V$ is of class $\cC^2$ and its derivative at $u$ is given by 
$$I_V'(u)v=\langle u,v\rangle_V - \io f(u)v,\qquad v\in D^{1,2}_0(\o).$$
The critical points of $I_V$ are the solutions to the problem \eqref{eq:problem}.

For $g\in G$ and $u\in D^{1,2}_0(\o)$ we define $gu\in D^{1,2}_0(\o)$ by $gu(x):=u(g^{-1}x)$. Since $\o$ and $V$ are $G$-invariant, $gu$ is well defined and $\langle gu,gv\rangle_V=\langle u,v\rangle_V$ for all $g\in G$ and $u,v\in D^{1,2}_0(\o)$. Therefore, the functional $I_V$ is $G$-invariant, i.e., $I_V(gu)=I_V(u)$, and, by the principle of symmetric criticality \cite[Theorem 1.28]{w}, the $G$-invariant solutions to \eqref{eq:problem} are the critical points of its restriction to the space
$$D^{1,2}_0(\o)^G:=\{u\in D^{1,2}_0(\o):gu=u\text{ for all }g\in G\},$$
which is the space of $G$-invariant functions in $D^{1,2}_0(\o)$. Abusing notation we write
$$I_V:D^{1,2}_0(\o)^G\to\r$$
for the restriction of $I_V$ to $D^{1,2}_0(\o)^G$.

Recall that a sequence $(u_k)$ in $D^{1,2}_0(\o)^G$ such that
$$I_V(u_k)\to c\qquad\text{and}\qquad I'_V(u_k)\to 0\text{ \ in \ }(D^{1,2}_0(\o)^G)'$$
is called a \emph{Palais-Smale sequence} for $I_V$ at the level $c$, and $I_V$ is said to satisfy $(PS)_c$ in $D^{1,2}_0(\o)^G$ if any such sequence contains a convergent subsequence.

\begin{lemma}\label{lem:ps bounded}
If $(u_k)$ is a Palais-Smale sequence for $I_V$ at the level $c$, then $(u_k)$ is bounded in $D^{1,2}_0(\o)^G$ and $c\geq 0$.
\end{lemma}

\begin{proof}
By assumption $(f_2)$,
\begin{align}\label{eq:bounded}
0&\leq\Big(\frac{1}{2}-\frac{1}{\theta}\Big)\|u_k\|^2_V\leq \Big(\frac{1}{2}-\frac{1}{\theta}\Big)\|u_k\|^2_V+\io \Big(\frac{1}{\theta}f(u_k)u_k -F(u_k)\Big) \nonumber \\
&=I_V(u_k)-\frac{1}{\theta}I'_V(u_k)u_k\leq |I_V(u_k)|+o(1)\|u_k\|_V.
\end{align}
Therefore, $(u_k)$ is bounded in $D^{1,2}_0(\o)^G$ and $c\geq 0$.
\end{proof}

Our aim is to prove the following result.

\begin{theorem} \label{thm:compactness}
\begin{itemize}
\item[$(a)$] If $\o$ is bounded, the functional $I_V$ satisfies $(PS)_c$ in $D^{1,2}_0(\o)^G$ for every $c\in\r$.
\item[$(b)$] If $\o$ is an exterior domain, then the functional $I_V$ satisfies $(PS)_c$ in $D^{1,2}_0(\o)^G$ for every 
$$c<\Big(\min_{x\in\rn\smallsetminus\{0\}}\#Gx\Big)c_\infty,$$
where $c_\infty$ is the ground state energy of the limit problem \eqref{eq:limit problem}. 
In particular, if $\#Gx=\infty$ for every $x\in\rn\smallsetminus\{0\}$, then $I_V$ satisfies $(PS)_c$ in $D^{1,2}_0(\o)^G$ for every $c\in\r$.
\end{itemize}
\end{theorem}

First, we state some lemmas that are needed for the proof.

\begin{lemma}\label{lem:cm} 
If $u_k\rh u$ weakly in $D^{1,2}(\rn)$ then, after passing to a subsequence, the following statements hold true:
\begin{itemize}
\item[$(a)$] $\|u_k\|_V^2 = \|u_k -u\|^2 + \|u\|_V^2 + o(1).$
\item[$(b)$] $\irn |f(u_k) - f(u)| |\vp|=o(1)$ \ for every $\vp\in\cC^{\infty}_{c}(\rn)$.
\item[$(c)$] $\irn F(u_k) = \irn F(u_k - u) + \irn F(u) + o(1)$.
\item[$(d)$] $\irn f(u_k)u_k = \irn f(u_k - u)(u_k-u) + \irn f(u)u + o(1)$.
\item[$(e)$] $Vu_k\to Vu$ \ and \ $f(u_k) - f(u_k -u) \to f(u)$  strongly in $(D^{1,2}_0(\o))'$.
\end{itemize}
\end{lemma}

\begin{proof}
Statements $(a),(b)$ and $(c)$ are proved in \cite[Lemma 3.8]{cm}. The proof of $(d)$ is obtained by replacing $F(s)$ by $f(s)s$ in the proof of $(c)$. Next, we prove $(e)$.

Let $\eps>0$. By assumption, $V\in L^{N/2}(\rn)\cap L^r(\rn)$ with $r>N/2$. We fix $\rho>0$ such that 
$$\int_{\rn\smallsetminus B_\rho}|V|^{N/2}<\eps^{N/2},$$ 
and define $\eta$ by $\frac{1}{r}+\frac{1}{\eta}+\frac{1}{2^*}=1$. Then, $\eta<2^*$ and, after passing to a subsequence, there exists $k_0\in\n$ such that 
$$\int_{B_\rho}|u_k-u|^\eta<\eps^\eta\qquad\text{for all \ }k\geq k_0.$$
As a consequence,
\begin{align*}
\Big|\irn V(u_k-u)v\Big|&\leq \int_{B_\rho}|V(u_k-u)v|+\int_{\rn\smallsetminus B_\rho}|V(u_k-u)v|\\
&\leq |V|_{L^r(\rn)}|u_k-u|_{L^\eta(B_\rho)}|v|_{L^{2^*}(\rn)} + |V|_{L^{N/2}(\rn\smallsetminus B_\rho)}|u_k-u|_{L^{2^*}(\rn)}|v|_{L^{2^*}(\rn)}\\
&\leq C\eps\|v\|\qquad\text{for every \ }v\in D^{1,2}(\rn)\text{ \ and \ }k\geq k_0.
\end{align*}
This shows that $V(u_k-u)\to 0$ \ in $(D^{1,2}_0(\o))'$. The proof that \ $f(u_k) - f(u_k -u) \to f(u)$ \ in $(D^{1,2}_0(\o))'$ is given in \cite[Lemma 3.8]{cm}. 
\end{proof}

The following version of Lions' vanishing lemma will play a crucial role.

\begin{lemma}\label{lem:lions} 
If $(u_k)$ is bounded in $D^{1,2}(\rn)$ and there exists $R>0$ such that 
$$\lim_{k\to\infty} \left( \sup_{y \in \rn} \int_{B_R(y)} |u_k|^2 \right)=0,$$
then $\lim_{k\to\infty} \irn f(u_k)u_k=0.$
\end{lemma}

\begin{proof}
See \cite[Lemma 3.5]{cm}.
\end{proof}

If $K$ is a closed subgroup of $G$, the \emph{homogeneous space} $G/K$ is the space of right cosets $gK$. Its cardinality is called the \emph{index of $K$ in $G$}, denoted $[G:K]$. We write $G\xi:=\{g\xi:g\in G\}$ for the \emph{$G$-orbit} of a point $\xi\in\rn$ and $G_{\xi}:=\{g\in G:g\xi=\xi\}$ for the \emph{isotropy subgroup} of $\xi$. It is readily seen that the map $G/G_\xi\to G\xi$ given by $gG_\xi\mapsto g\xi$ is a $G$-homeomorphism. In particular $\#G\xi=[G:G_\xi]$.

\begin{lemma} \label{lem:K} 
Given a sequence $(y_k)$ in $\rn$ there exists a sequence $(\zeta_k)$ in $\rn$ and a closed subgroup $K$ of $G$ such that for some subsequence of $(y_k)$, denoted in the same way, the following hold:
\begin{itemize}
\item [$(a)$]The sequence $(\mathrm{dist} (Gy_k,\zeta_k))$ is bounded.
\item[$(b)$] $G_{\zeta_k}= K$ for all $k\in\n$.
\item[$(c)$] If $[G:K]<\infty$, then $|g\zeta_k - \what g \zeta_k|\to\infty$ for any $g,\what g\in G$ with $gK\neq \what g K$.
\item[$(d)$] If $[G:K]=\infty$ then, for any given $m\in\n$, there exist $g_1,\ldots,g_m\in G$ such that $|g_i\zeta_k - \what g_j\zeta_k|\to\infty$ if $i\neq j$.
\end{itemize}
\end{lemma}

\begin{proof}
See \cite[Lemma 3.2]{ccs}.
\end{proof}

The following lemma gives the main step in the proof of Theorem \ref{thm:compactness}.

\begin{lemma}\label{lem:spliting} 
Let $V=0$ and $(u_k)$ be a Palais-Smale sequence for $I_0$ in $D^{1,2}_0(\o)^G$ at the level $c$ such that  $u_k\rh 0$ weakly in $D^{1,2}_0(\o)$. 
\begin{itemize}
\item[$(a)$] If $\o$ is bounded, then $u_k\to 0$ strongly in $D^{1,2}_0(\o)$. 
\item[$(b)$] If $\o$ is an exterior domain and $(u_k)$ does not converge strongly to $0$ in $D^{1,2}_0(\o)$ then, after passing to a subsequence, there exist a sequence $(\zeta_k)$ in $\rn\smallsetminus\{0\}$, a closed subgroup $K$ of finite index in $G$, and a nontrivial solution $w$ to the limit problem \eqref{eq:limit problem} such that
$$G_{\zeta_k}= K\text{ \ for all \ }k \in \mathbb{N}\qquad\text{and}\qquad c\geq[G:K] \, I_{\infty}(w),$$ 
where $I_\infty:D^{1,2}(\rn)\to\r$,
$$I_\infty(w):=\frac{1}{2}\|w\|^2-\irn F(w),$$
is the energy functional of the limit problem \eqref{eq:limit problem}.  
\end{itemize}
\end{lemma}

\begin{proof}
Assume that $(u_k)$ does not converge strongly to $0$ in $D^{1,2}_0(\o)$. Then there exist $C_0>0$ and a subsequence such that $\|u_k\|\geq C_0$ for all $k\in\n$. As $(u_k)$ is bounded in $D^{1,2}_0(\o)$ and $I'_0(u_k)\to 0$, we have that
$$o(1)=I_0'(u_k)u_k= \|u_k\|^2 - \io f(u_k)u_k.$$
Hence,
$$0<C_0^2\leq\|u_k\|^2=\io f(u_k)u_k+o(1),$$
and, by Lemma \ref{lem:lions}, there exist $\delta >0$ and a sequence $(y_k)$ in $\rn$ such that
$$\int_{B_1(y_k)} |u_n|^2 = \sup_{y\in\rn} \int_{B_1(y)} |u_k|^2 \geq\delta.$$
For the sequence $(y_k)$, we fix a sequence $(\zeta_k)$ in $\rn$ and a closed subgroup $K$ of $G$ with the properties stated in Lemma \ref{lem:K}. In particular, there exist $g_k\in G$ and $C>0$ such that $|g_k^{-1}\zeta_k-y_k|=\mathrm{dist}(G\zeta_k,y_k)\leq C$. As $u_k$ is $G$-invariant, we get
\begin{equation}\label{eq:nontrivial}
\int_{B_{C+1}(\zeta_k)} |u_k|^{2} \geq \int_{B_1(g_k y_k)} |u_k|^{2} = \int_{B_1(y_k)} |u_k|^{2}\geq\delta>0.
\end{equation}
Set $w_k := u_k(\cdot + \zeta_k)$. Since $(w_k)$ is bounded in $D^{1,2}(\rn)$, passing to a subsequence, we have that
$w_k\rh w$ weakly in $D^{1,2}(\rn)$, $w_k\to w$ in $L^{2}_{\mathrm{loc}}(\rn)$ and $w_k\to w$ a.e. in $\rn$.
The inequality \eqref{eq:nontrivial} yields
$$\int_{B_{C+1}(0)} |w_k|^{2} \geq \delta>0.$$
Therefore, $w\neq 0$. Then, as $u_k \rh 0$ weakly in $D^{1,2}(\rn)$, an easy argument shows that $|\zeta_k|\to\infty$.

The inequality \eqref{eq:nontrivial} implies that $\dist(\zeta_k,\o)<C+1$ for all $k$. Thus, if $\o$ is bounded, then $(\zeta_k)$ is bounded and we obtain a contradiction. This proves $(a)$.

Assume now that $\o$ is an exterior domain. Given $\vp\in\cC^\infty_c(\rn)$ we set $\vp_k(x):=\vp(x-\zeta_k)$. Since $|\zeta_k|\to\infty$, we have that $\vp_k\in \cC^\infty_c(\o)$ for $k$ large enough. Then, using Lemma \ref{lem:cm}$(b)$ one sees that
$$I'_\infty(w)\vp=I'_\infty(w_k)\vp+o(1)=I'_0(u_k)\vp_k+o(1)=o(1).$$
Hence, $w$ is a nontrivial solution to the limit problem \eqref{eq:limit problem}.

Assume there exist $g_1,\ldots, g_m \in G$ such that $|g_i\zeta_k- g_j\zeta_k|\to\infty$ if $i \neq j$. Then, for each $j\in\{1,\ldots,m\}$,
$$g_j w_k - \sum_{i=j+1}^{m} (g_i w)(\cdot - g_i \zeta_k + g_j \zeta_k) \rh g_j  w \quad \text{weakly in \ } D^{1,2}(\rn),$$
with the sum being $0$ if $j=m$. Lemma \ref{lem:cm}$(c)$ yields
\begin{align*}
&\irn F\Big(g_jw_k - \sum_{i=j+1}^{m} (g_iw)(\cdot - g_i \zeta_k + g_j \zeta_k) \Big)\\
&\qquad= \irn F\Big(g_j w_k - \sum_{i=j}^{m} (g_iw)(\cdot - g_i \zeta_k + g_j \zeta_k)\Big) + \irn F(g_j w) + o(1).
\end{align*}
Since $u_k$  is $G$-invariant, performing the change of variable $x \mapsto x - g_j \zeta_k$ we derive
\begin{align*}
&\irn F\Big(u_k - \sum_{i=j+1}^{m} (g_i w)(\cdot - g_i \zeta_k ) \Big)= \irn F\Big(u_k - \sum_{i=j}^{m} (g_iw)(\cdot - g_i \zeta_k ) \Big) + \irn F(w) + o(1).
\end{align*}
Iterating this identity for $j= 1,\ldots, m$ we obtain
\begin{equation} \label{eq:F}
\irn F(u_k)=\irn F(u_k - z_k)+m\irn F(w) + o(1),
\end{equation}
where
$$z_k:=\sum_{i=1}^{m} (g_iw)(\cdot - g_i \zeta_k).$$
Similarly, using statements $(d)$ and $(a)$ of Lemma \ref{lem:cm} we obtain
\begin{equation} \label{eq:f}
\irn f(u_k)u_k=\irn f(u_k - z_k)(u_k - z_k)+m\irn f(w)w+o(1).
\end{equation}
and
\begin{equation}\label{eq:norm}
\|u_k\|^2 = \|u_k - z_k\|^2 + m \|w\|^2 + o(1).
\end{equation}
Since $w$ solves \eqref{eq:limit problem}, from \eqref{eq:norm} and \eqref{eq:f} we get
$$o(1)=I'_0(u_k)u_k=\|u_k-z_k\|^2-\irn f(u_k-z_k)(u_k-z_k)+o(1),$$
and from assumption $(f_2)$ we derive
$$\frac{1}{2}\|u_k-z_k\|^2-\irn F(u_k-z_k)=\irn\Big(\frac{1}{2}f(u_k-z_k)(u_k-z_k)-F(u_k-z_k)\Big)+o(1)\geq 0+o(1).$$
This inequality, combined with \eqref{eq:norm} and \eqref{eq:F}, yields
\begin{equation}\label{eq:lower bound}
c+o(1)=I_0(u_k)=\frac{1}{2}\|u_k-z_k\|^2-\irn F(u_k-z_k)+mI_\infty(w)+o(1)\geq mI_\infty(w)+o(1).
\end{equation}
Therefore, $m$ cannot be arbitrarily large. It follows from Lemma \ref{lem:K} that $[G:K]<\infty$. Setting $m:=[G:K]$ and passing to the limit in \eqref{eq:lower bound} we obtain
$$c\geq [G:K]\,I_\infty(w).$$
This completes the proof of $(b)$.
\end{proof}

\medskip

\begin{proof}[Proof of Theorem \ref{thm:compactness}]
Let $(u_k)$ be a Palais-Smale sequence for $I_V$ in $D^{1,2}_0(\o)^G$ at the level $c$. By Lemma \ref{lem:ps bounded}, $(u_k)$ is bounded in $D^{1,2}_0(\o)$ and, passing to a subsequence, $u_k\rh u$ weakly in $D^{1,2}_0(\o)^G$. Using Lemma \ref{lem:cm}$(b)$ one sees that $u$ is a solution to \eqref{eq:problem}. Hence, by $(f_2)$,
$$I_V(u)=I_V(u)-I'_V(u)u=\Big(\frac{1}{2}-\frac{1}{\theta}\Big)\|u\|^2_V+\io\Big(\frac{1}{\theta}f(u)u-F(u)\Big)\geq 0.$$
Set $v_k:=u_k-u$. Then $v_k\rh 0$ weakly in $D^{1,2}_0(\o)$ and statements $(a),(c)$ and $(e)$ of Lemma \ref{lem:cm} yield
$$I_0(v_k)\to d:=c-I_V(u)\qquad\text{and}\qquad I'_0(v_k)\to 0\text{ \ in \ }(D^{1,2}_0(\o)^G)'.$$

If $\o$ is bounded, Lemma \ref{lem:spliting}$(a)$ states that $v_k\to 0$ strongly in $D^{1,2}_0(\o)$, i.e., $u_k\to u$ strongly in $D^{1,2}_0(\o)$. This proves statement $(a)$.

If $\o$ is an exterior domain and $(v_k)$ does not converge strongly to $0$ in $D^{1,2}_0(\o)$ then, by Lemma \ref{lem:spliting}$(b)$, there exist a closed subgroup $K$ of finite index in $G$, a sequence $\zeta_k$ in $\rn\smallsetminus\{0\}$ such that $\#G\zeta_k=[G:K]$ and a nontrivial solution $w$ to the limit problem \eqref{eq:limit problem} such that
$$c\geq d\geq[G:K]\,I_{\infty}(w)\geq(\#G\zeta_k)c_\infty.$$
Therefore, if $\o$ is an exterior domain and
$$c<\big(\min_{x\in\rn\smallsetminus\{0\}}\#Gx\big)c_\infty,$$
then $v_k\to 0$ strongly in $D^{1,2}_0(\o)$, i.e., $u_k\to u$ strongly in $D^{1,2}_0(\o)$. This proves statement $(b)$.
\end{proof}

\section{A variational principle for sign-changing solutions}

Let $G$ be a closed subgroup of $O(N)$. We assume throughout this section that $\o$ and $V$ are $G$-invariant and that assumptions $(V_1)$ and $(f_1)-(f_3)$ hold true.

\subsection{The structure of the Nehari manifold}

The nontrivial $G$-invariant critical points of the functional $I_V$ belong to the set
$$\cN(\o)^G:= \{u\in D^{1,2}_0(\o)^G : u\neq 0, \ I'_V(u)u=0\}.$$
Define
$$c(\o)^G:= \inf_{u\in\cN(\o)^G} I_V(u).$$ 
Before stating the properties of $\cN(\o)^G$, we note that assumptions $(f_1)$ and $(f_2)$ guarantee that $I_V$ has the mountain pass geometry in bounded domains.

\begin{lemma}\label{lem:mountain pass}
\begin{itemize}
\item[$(a)$] There exist $r>0$ and $a>0$ such that $I_V(u)>0$ if $\|u\|_V\leq r$ and $u\neq 0$, and $I_V(u)\geq a$ if $\|u\|_V=r$.
\item[$(b)$] If $\Lambda$ is a bounded $G$-invariant open subset of $\o$ and $W\subset D^{1,2}_0(\Lambda)^G$ is a finite dimensional linear subspace, then there exists $R>r$ such that $I_V(w)\leq 0$ for every $w\in W$ with $\|w\|_V\geq R$.
\end{itemize}
\end{lemma}

\begin{proof}
$(a):$ \ It follows from $(f_1)$ that $F(s)\leq A_1|s|^{2^*}$ for all $s\in \r$. Therefore, using Sobolev's inequality,
$$I_V(u)\geq\frac{1}{2}\|u\|_V^2-A_1|u|_{2^*}^{2^*}\geq\frac{1}{2}\|u\|_V^2-C\|u\|_V^{2^*}\qquad\text{for all \ }u\in D^{1,2}_0(\o)^G$$
and some $C>0$. This implies $(a)$.

$(b):$ \ It follows from $(f_2)$ that $F(s)\geq 0$ and $F(s)\geq a_1|s|^\theta-a_2$ for all $s\in \r$ and some constants $a_1,a_2>0$. Since $\Lambda$ is bounded, we get
$$I_V(u)\leq\frac{1}{2}\|u\|_V^2-\io F(u)\leq\frac{1}{2}\|u\|_V^2-\int_{\Lambda} F(u)\leq\frac{1}{2}\|u\|_V^2-a_1\int_{\Lambda}|u|^\theta+a_2|\Lambda|\qquad\text{for all \ }u\in D^{1,2}_0(\o)^G,$$
and, as $W$ has finite dimension, 
$$I_V(w)\leq \frac{1}{2}\|w\|_V^2-\what C\|w\|_V^\theta+a_2|\Lambda|\qquad\text{for all \ }w\in W$$
and some $\what C>0$. Since $\theta>2$, this implies $(b)$.
\end{proof}

A subset $\cY$ of $D^{1,2}_0(\o)^G$ will be called \emph{symmetric} if $-u\in\cY$ for every $u\in\cY$.

\begin{lemma}\label{lem:nehari}
\begin{itemize} 
\item [$(i)$] The exists $\vr>0$ such that $\| u\| \geq \vr$ for every $u\in\cN(\o)^G$.
\item [$(ii)$] $\cN(\o)^G$ is a closed $\cC^1$-submanifold of codimension $1$ of $D^{1,2}_0(\o)^G$, called the \emph{Nehari manifold}.
\item [$(iii)$] $\cN(\o)^G$ is a natural constraint for the functional $I_V$.
\item [$(iv)$] $c(\o)^G>0$.
\item[$(v)$] If $u\in D^{1,2}_0(\o)^G$ and $u\neq 0$, then there exists $t_u\in(0,\infty)$ such that $t_uu\in \cN(\o)^G$.
\item[$(vi)$] If $u \in\cN(\o)^G$, then $t_u=1$ and the function $t \mapsto I_V(tu) $ is strictly increasing in $[0,1)$ and strictly decreasing in $(1,\infty)$. In particular,
$$I_V(u)= \max_{t>0} I_V(tu).$$
\item[$(vii)$] $\cN(\o)^G$ is symmetric.
\end{itemize}
\end{lemma}

\begin{proof}
The proof of items $(i)-(iv)$ is given in \cite[Lemma 3.2]{cm}.

$(v):$ \ Given $u\in D^{1,2}_0(\o)^G$ and $u\neq 0$, consider the function $\sigma_u(t):=I_V(tu)$, $t\in[0,\infty)$. It follows from Lemma \ref{lem:mountain pass}$(a)$ that $\sigma_u(t_0)\geq a>0$ for some $t_0\in(0,\infty)$. Now fix a bounded $G$-invariant open subset $\Lambda$ of $\o$ such that $u\neq 0$ in $\Lambda$. As in the proof of Lemma \ref{lem:mountain pass}$(b)$ we have
$$I_V(tu)\leq\frac{1}{2}\|tu\|_V^2-a_1\int_{\Lambda}|tu|^\theta+a_2|\Lambda|=\Big(\frac{1}{2}\|u\|_V^2\Big)t^2-\Big(a_1\int_{\Lambda}|u|^\theta\Big) t^\theta+a_2|\Lambda|\qquad\text{for all \ }t\in[0,\infty).$$
Hence, $\sigma_u(t)\to -\infty$ as $t\to\infty$. Therefore, $\sigma_u$ attains its maximum at some $t_u\in(0,\infty)$. As a consequence,
$$0=t_u\sigma_u'(t_u)=\|t_uu\|_V^2-\io f(t_uu)t_uu,$$
i.e., $t_uu\in \cN(\o)^G$. This proves $(v)$.

$(vi):$ \ If $t_0\in(0,\infty)$ is a critical point of $\sigma_u$, then, by $(f_2)$,
$$t_0^2\sigma_u''(t_0)=\|t_0u\|_V^2-\io f'(t_0u)(t_0u)^2=\io (f(t_0u)t_0u-f'(t_0u)(t_0u)^2)<0.$$
Hence, $\sigma_u''(t_0)<0$. This shows that every critical point of $\sigma_u$ is a strict local maximum. Therefore, $t_u$ is the only critical point and statement $(vi)$ follows. 

$(vii)$ is an immediate consequence of assumption $(f_3)$.
\end{proof}

As a consequence of the previous lemma, the complement of the Nehari manifold is the disjoint union of two open symmetric sets
\begin{equation} \label{eq:complement of nehari}
D^{1,2}_0(\o)^G\smallsetminus \cN(\o)^G = \cB_0 \cup \cB_\infty,
\end{equation}
with $\cB_0:=\{tu:u\in\cN(\o)^G, \ t\in[0,1)\}$ and $\cB_\infty:=\{tu:u\in\cN(\o)^G, \ t\in(1,\infty)\}$, and 
\begin{equation} \label{eq:I>0 on B0}
I_V(u)>0\quad\text{for every \ }u\in\cB_0\smallsetminus\{0\}.
\end{equation}

The $G$-invariant sign-changing critical points of $I_V$ belong to the set
$$\cE(\o)^G:= \{u\in\cN(\o)^G : u^+\in\cN(\o)^G\text{ and }u^-\in\cN(\o)^G\},$$
where $u^+:=\max\{u,0\}$ and $u^-:=\min\{u,0\}$. It has the following properties.

\begin{lemma} \label{lem:N-E}
\begin{itemize}
\item[$(i)$] $\cE(\o)^G$ is closed and symmetric.
\item[$(ii)$] The complement of $\cE(\o)^G$ in $\cN(\o)^G$ has two connected components $\cU$ and $-\cU$.
\end{itemize}
\end{lemma}

\begin{proof}
$(i):$ \ It is clear that $\cE(\o)^G$ is closed and that $u\in\cE(\o)^G$ iff $-u\in\cE(\o)^G$.

$(ii):$ Let
$$\Psi(u):=\|u\|^2_V-\io f(u)u.$$
The complement of $\cE(\o)^G$ in $\cN(\o)^G$ is the union of the sets
\begin{align*}
\cW^+:=\{w\in\cN(\o)^G:w\geq 0\}\cup \{w\in\cN(\o)^G:\Psi(w^+)<0\},\\
\cW^-:=\{w\in\cN(\o)^G:w\leq 0\}\cup \{w\in\cN(\o)^G:\Psi(w^+)>0\}.
\end{align*}
As in \cite[Lemma 2.5]{ccn} one sees that $\cW^+$ and $\cW^-$ are open and connected. Note that, for every $w\in\cN(\o)^G$,
$$\|w^+\|^2_V+\|w^-\|^2_V=\|w\|^2_V=\io f(w)w=\int_{w>0} f(w)w+\int_{w<0} f(w)w=\io f(w^+)w^++\io f(w^-)w^-.$$
Therefore,
$$\Psi(w^+)<0\text{ \ iff \ }\Psi(w^-)>0.$$
It follows from assumption $(f_3)$, that $\Psi(w^-)=\Psi(-w^-)=\Psi((-w)^+)$. This shows that $w\in \cW^+$ iff $-w\in\cW^-$, that is, $\cW^-=-\cW^+$.
\end{proof}

\subsection{A mountain pass theorem for sign-changing solutions}

Let $u\in D^{1,2}(\o)^G $. The gradient of $I_V$ at $u$ is $\nabla I_V(u) = u - Q(u)$ where $Q(u)$ is the unique element in $D^{1,2}(\o)^G$ such that 
\begin{equation}\label{eq:gradient}
\langle Q(u), v \rangle_V= \io f(u)v \qquad  \text{for all \ }v\in D^{1,2}_0(\o)^G.
\end{equation}
Consider the negative gradient flow $\eta:\mathcal{G}\to D^{1,2}_0(\o)^G$ of $I_V$, defined by
\begin{equation*}
\begin{cases}
\frac{\d}{\d t}\eta(t,u)= - \nabla I_V (\eta(t,u)),\\
\eta(0,u)= u,
\end{cases}
\end{equation*}
where $\mathcal{G}:= \{ (t,u): u \in D^{1,2}_0(\o)^G, \ 0 \leq t < T(u) \}$ and $T(u) \in ( 0, \infty]$ is the maximal existence time for the trajectory $t \to \eta(t,u)$. A subset $\cZ$ of $D^{1,2}_0(\o)^G$ is said to be \emph{strictly positively invariant} under $\eta$ if 
$$\eta(t,u) \in \mathrm{int} (\cZ) \quad \text{for every \ } u \in \cZ \text{ and every } t \in (0,T(u)),$$
where $\mathrm{int}(\mathcal{Z})$ denotes the interior of $\mathcal{Z}$ in $D^{1,2}_0(\Omega)^G$. If $\cZ$ is strictly positively invariant under $\eta$, then the set
$$\cA(\cZ):= \{ u \in D^{1,2}_0(\o)^G: \eta(t,u) \in \cZ\text{ for some } t \in (0,T(u)) \}$$
is open in $D^{1,2}_0(\o)^G$ and the entrance time map $e_{\cZ}: \cA(\cZ)\to\r$ defined by
$$e_{\cZ}(u):= \inf \{ t \geq 0:\eta(t,u)\in\cZ\}$$
is continuous. We write 
$$\cP^G:= \{u\in D^{1,2}_0(\o)^G:u\geq 0\}$$
for the convex cone of non-negative functions in $D^{1,2}_0(\o)^G$ and, for $\alpha > 0$, we set 
$$B_{\alpha}(\cP^G):= \{u\in D^{1,2}_0(\o)^G:\mathrm{dist}(u,\cP^G)\leq\alpha\},$$
where $\mathrm{dist}(u,\cA):= \inf_{v \in \mathcal{A}} \|u-v\|_V$.

\begin{lemma}\label{lem:equator} There exists $\alpha >0$ such that
\begin{itemize}
\item [$(a)$] $[B_{\alpha} (\cP^G) \cup B_{\alpha}(- \cP^G)] \cap \cE^G = \emptyset$, and
\item [$(b)$] $B_{\alpha} ( \cP^G)$ and $B_{\alpha}(- \cP^G)$ are strictly positively invariant under $\eta$.
\end{itemize}
\end{lemma}

\begin{proof} 
$(a):$ \ For any $u \in D^{1,2}_0(\o)^G$ the Sobolev inequality yields a positive constant $C$ such that
\begin{equation} \label{eq:dist}
|u^{-}|_{2^*} = \inf_{v \in \mathcal{P}^G} |u-v|_{2^*} \leq C \inf_{v \in \cP^G} \|u-v \|_V = C \,\mathrm{dist}(u, \mathcal{P}^G).
\end{equation}
If $u\in\cE^G$, then $u^-\in \cN^G$ and, by Lemma \ref{lem:nehari}$(i)$ and assumption $(f_1)$, 
$$0<\vr^2\leq\|u^-\|_V^2 = \io f(u^-)u^- \leq A_1\io|u^-|^{2^*} = A_1|u^-|_{2^*}^{2^*}.$$
Hence, there exists $\alpha>0$ such that $\alpha < \mathrm{dist}(u,\cP^G )$ for all $u \in \cE^G$. \ Since $\cE^G$ is symmetric, this implies that $\alpha < \mathrm{dist}(u,-\cP^G )$ for all $u \in \cE^G$.

$(b):$ \ Using \eqref{eq:gradient}, $(f_2)$, $(f_1)$, the Hölder and the Sobolev inequalities, and \eqref{eq:dist} we obtain
\begin{align*}
\mathrm{dist} (Q(u), \cP^G) \|Q(u)^{-}\|_V &\leq \|Q(u)^{-}\|_V^2 = \langle Q(u), Q(u)^- \rangle_V = \int_{\Omega} f(u) Q(u)^- \\
&= \int_{u>0} f(u) Q(u)^- + \int_{u<0} f(u) Q(u)^- \leq \int_{\Omega} f(u^-) Q(u)^-\\
&\leq A_1\int_{\Omega} |u^-|^{2^*-1} Q(u)^{-} \leq A_1|u^-|_{2^*}^{2^*-1}|Q(u^-)|_{2^*}\\
&=C\,\mathrm{dist}(u, \cP^G )^{2^*-1} \|Q(u)^-\|_V.
\end{align*}
If $Q(u)^-\neq 0$, then 
$$\text{dist}(Q(u),\cP^G) \leq C\, \mathrm{dist}(u,\cP^G)^{2^*-1}. $$
So, setting $\alpha < \frac{1}{2}C^{-1/(2^*-2)}$ and $\delta:=(\frac{1}{2})^{2^*-2} \in (0,1)$, we have that
$$\mathrm{dist} (Q(u),\cP^G) \leq \delta\,\mathrm{dist} (u, \cP^G) \qquad \text{for every \ } u \in B_{\alpha}(\cP^G).$$
It follows that $Q(u)\in\mathrm{int}(B_{\alpha}(\cP^G))$ if $u\in B_{\alpha}(\cP^G).$ Since $B_\alpha(\cP^G)$ is closed and convex, applying \cite[Theorem 5.2]{d} we conclude that
\begin{equation} \label{eq:invariance}
\eta(t,u) \in B_\alpha(\cP^G) \quad \text{for every }u\in B_\alpha(\cP^G)\text{ \ and \ } t \in [0,T(u)).
\end{equation}
Now, arguing by contradiction, assume that $\eta(t,u) \in \partial(B_{\alpha}(\cP^G))$ for some $u \in B_{\alpha}(\cP^G)$ and $t \in (0,T(u))$. Then, by Mazur's separation theorem, there exist a continuous linear functional $\mathscr{L} \in (D^{1,2}_0(\o)^G)'$ and $\beta >0$ such that $\mathscr{L}(\eta(t,u))= \beta$ and $\mathscr{L}(u)> \beta$ for every $u \in \text{int}(B_{\alpha}(\mathcal{P}^G))$.
It follows that
$$\frac{d}{dt} \Big|_{s=t}\mathscr{L}(\eta(s,u))=\mathscr{L}(- \nabla I_V(\eta(t,u)))= \mathscr{L}(Q(\eta(t,u))) - \beta >0.$$
Hence, there exists $\eps>0$ such that $\mathscr{L} (\eta(s,u))<\beta$ for every $s \in( t-\eps,t)$. Then, $\eta(s,u) \notin B_{\alpha}(\cP^G)$, which is a  contradiction. This proves that $\eta(t,u) \in \mathrm{int}(B_{\alpha}(\cP^G))$ for every $u\in B_\alpha(\cP^G)$ and $0<t < T(u)$, as claimed.
\end{proof}

Fix $\alpha$ as in the previous lemma. Given $d\in\r$, set $I^d_V:=\{u\in D^{1,2}_0(\o)^G:J(u)\leq d\}$ and define
\begin{equation*}
\cZ_d^G:=B_{\alpha} (\cP^G) \cup B_{\alpha}(- \cP^G)\cup I^d_V.
\end{equation*}
Lemma \ref{lem:equator} yields the following result.

\begin{corollary} \label{cor:retraction}
If $I_V$ does not have a sign-changing critical point $u \in D^{1,2}_0(\o)^G$ with $I_V(u)=d$, then the set $\cZ_d^G$ is strictly positively invariant under $\eta$, and the map
$$\varrho_d:\cA(\cZ_d^G)\to\cZ_d^G, \qquad \varrho_d(u)= \eta(e_{\cZ_d^G} (u),u),$$
is odd and continuous, and satisfies $\varrho_d(u)=u$ for every $u\in\cZ_d^G$.
\end{corollary}

We introduce a suitable topological invariant as follows.

\begin{definition}\label{def:genus}
Let $\cZ\subset\cY$ be symmetric subsets of $D^{1,2}_0(\o)^G$. The \emph{genus of $\cY$ relative to $\cZ$}, denoted $\gen(\cY,\cZ)$, is the smallest number $m$ such that $\cY$ can be covered by $m+1$ open symmetric subsets $\cU_0$, $\cU_1, \ldots \cU_m$ of $D^{1,2}_0(\o)^G$ with the following properties:
\begin{itemize}
\item [$(i)$] $\cZ\subset\cU_0$ and there exists an odd continuous map $\vartheta_0 : \cU_0\to\cZ$ such that $\vartheta_0(u)=u$ for all $u\in\cZ$.
\item [$(ii)$] There exist odd continuous maps $\vartheta_j: \cU_j \to\{-1,1\}$ for each $j=1,\ldots,m.$
\end{itemize}
If no such cover exists, we set $\gen(\cY,\cZ):=\infty$.
\end{definition}

Recall that $I_V:D^{1,2}_0(\o)^G\to\r$ is said to satisfy $(PS)_c$ if every sequence $(v_k)$ in $D^{1,2}_0(\o)^G$ such that
$$I_V(v_k)\to c\qquad\text{and}\qquad \nabla I_V(u_k)\to 0,$$
contains a convergent subsequence.

\begin{theorem} \label{thm:genus}
Let $d\geq 0$. If $I_V$ satisfies $(PS)_c$ at every $c\leq d$, then $I_V$ has at least $\gen(\cZ_d^G,\cZ_0^G)$ pairs of sign-changing critical points $\pm v\in D^{1,2}_0(\o)^G$ with $I_V(v)\leq d$.
\end{theorem}

\begin{proof}
The proof is formally identical to that of \cite[Proposition 3.6]{cp}, using now Corollary \ref{cor:retraction}.
\end{proof}

We are ready to state the main result of this section.

\begin{theorem} \label{thm:nodal principle}
Let $\Lambda$ be a bounded $G$-invariant open subset of $\o$ and $W$ be a finite dimensional linear subspace of $D^{1,2}_0(\Lambda)^G$. If $I_V$ satisfies $(PS)_c$ for every $c\leq \sup_{u\in W}I_V(u)$, then $I_V$ has at least $m:=\dim(W)$ pairs of critical points $\pm u_1,\ldots,\pm u_m$ in $D^{1,2}_0(\o)^G$ such that $u_1>0$, $u_j\in\cE(\o)^G$ if $j=2,\ldots,m$, 
$$I_V(u_1)=c(\o)^G\qquad\text{and}\qquad I_V(u_j)\leq \sup_{u\in W}I_V(u)\quad\text{for all \ }j=1,\ldots,m.$$
\end{theorem}

\begin{proof}
Let $m:=\dim(W)\geq 1$. By Lemma \ref{lem:nehari}, $W\cap\cN(\o)^G\neq\emptyset$. Hence, $d:=\sup_{u\in W}I_V(u)\geq c(\o)^G$. Therefore, $I_V$ satisfies $(PS)_c$ at $c=c(\o)^G$ and a standard argument shows that this value is attained at some positive function $u_1\in\cN(\o)^G$.

Next, adapting the argument in \cite[Theorem 3.7]{cp}, we will show that $k:=\gen(\cZ_d^G,\cZ_0^G)\geq m-1$. 

Let $\cU_0,\cU_1,\ldots,\cU_k$ be open symmetric subsets of $D^{1,2}_0(\o)^G$ such that $\cZ^G_d\subset\cU_0\cup\,\cU_1\cup\cdots\cup\,\cU_k$ and $\cU_0 \supset \cZ^G_0$, $\vartheta_0: \cU_0\to\cZ^G_0$ be an odd continuous map such that $\vartheta_0(u)=u$ for all $u \in \cZ^G_0$, and $\vartheta_j: \cU_j \to \{1,-1\}$ be odd continuous maps, $j =1,\ldots,k$. Making $\cU_0$ a little smaller, we may assume that $\vartheta_0$ is defined on the closure of $\cU_0$, and then extend it to an odd map $\what{\vartheta}_0: D^{1,2}_0(\o)^G\to D^{1,2}_0(\o)^G$ using Tietze's extension theorem.

Let $\cB_0$ be the connected component of the complement of $\cN(\o)^G$ that contains $0$, defined in \eqref{eq:complement of nehari}, and set
$$\cO:= \{w\in W:\what{\vartheta}_0(w)\in\cB_0\}.$$
Then $\cO$ is an open symmetric neighborhood of $0$ in $W$. Lemma \ref{lem:mountain pass} states that $I_V(w)\leq 0$ if $w\in W$ and $\|w\|_V\geq R$ for some $R>0$. Hence, if $\|w\|_V\geq R$, it follows from \eqref{eq:I>0 on B0} that $\what{\vartheta}_0(w)=w\notin \cB_0$. This shows that  $\cO$ is bounded. 

Define $U_j := \cU_j \cap \partial \cO$. Then $U_0, U_1, \ldots ,U_k$ are symmetric, open in $ \partial \cO$, and they cover $\partial \cO$. Moreover, by Lemma \ref{lem:equator},
$$\vartheta_0(U_0) \subset \cZ^G_0 \cap \cN(\o)^G \subset \cN(\o)^G \smallsetminus \cE(\o)^G.$$
It follows from Lemma \ref{lem:N-E}$(ii)$ that there exists an odd map $\eta : \cN(\o)^G \smallsetminus \cE(\o)^G \to\{1,-1\}$. Let $\eta_j:U_j \to\{1,-1\}$ be the restriction of the map $\eta_0 := \eta \circ \vartheta_0$ if $j=0$, and of $\eta_j := \vartheta_j,$ if $j= 1, \ldots,k$. Choose a partition of unity $\{\pi_j:\partial \mathcal{O} \to [0,1]: j =0,1, \ldots,k\}$ subordinated to the cover $\{U_0, U_1, \ldots,U_k \}$ consisting of even functions, and let $\{ e_1, \ldots, e_{k+1} \}$ be the canonical basis of $\r^{k+1}$. The map $\psi: \partial \mathcal{O} \rightarrow \r^{k+1}$ given by
$$\psi(u) = \sum_{j=0}^{k} \eta_j(u) \pi_j(u) e_{j+1}$$
is odd and continuous, and satisfies $\psi(u) \neq 0$ for every $u \in \partial \mathcal{O}$. By the Borsuk-Ulam theorem, $\dim(W) \leq k+1$, as claimed.
        
Applying Theorem \ref{thm:genus} we obtain $m-1$ pairs $\pm u_2,\ldots,\pm u_m$ of sign-changing critical points of $I_V$ with $I_V(u_j)\leq d$ for all $j=2,\ldots,m$. This completes the proof.
\end{proof}

\section{The proof of the main results}

\begin{proof}[Proof of Theorem \ref{thm:main1}]
Whether $\o$ is bounded or $\o$ is an exterior domain and $\#Gx=\infty$ for all $x\in\rn\smallsetminus\{0\}$, Theorem \ref{thm:compactness} states that $I_V$ satisfies $(PS)_c$ in $D^{1,2}_0(\o)^G$ for every $c\in\r$. As $D^{1,2}_0(\o)^G$ has infinite dimension, Theorem \ref{thm:nodal principle} yields $m$ pairs $\pm u_1,\ldots,\pm u_m$ of $G$-invariant solutions to \eqref{eq:problem} such that $u_1$ is positive, $I_V(u_1)=c(\o)^G$, and $u_2,\ldots,u_m$ change sign, for every $m\in\n$.
\end{proof}
\smallskip

\begin{proof}[Proof of Theorem \ref{thm:main2}]
Fix $R>0$ and set $\Theta := \R^N \smallsetminus B_R$. Let $\mathcal{P}_1$ be the set of all nonempty $O(N)$-invariant bounded domains (i.e., all open annuli) contained in $\Theta$ and, for each $k\geq 2$, let
\[
\mathcal{P}_k:=  \big\lbrace \{\Theta_1, \ldots, \Theta_k\}: \Theta_i \in \mathcal{P}_1, \ \Theta_i \cap \Theta_j = \emptyset  \text{ if } i \neq j \big\rbrace .
	\]
It follows from Theorem \ref{thm:main1} that in any $O(N)$-invariant domain $\Theta'\subset\Theta$ the problem
	\[
	 - \Delta u + V(x)u = f(u), \qquad u \in D^{1,2}_0(\Theta')^{O(N)},
	\] 
has a least energy solution $\omega_{\Theta'} \in \mathcal{N}(\Theta')^{O(N)} $. Define
	\[
	c_k:= \inf\left\lbrace \sum_{i=1}^{k}  I_V(\omega_{\Theta_i}): \{\Theta_1, \ldots , \Theta_k\} \in \mathcal{P}_k \right\rbrace.
	\]
As $\Theta_i \subset\Theta$, we have that $\cN(\Theta_i)^{O(N)}\subset\cN(\Theta)^{O(N)}$. Hence, $I_V(\omega_{\Theta_i})\geq I_V(\omega_\Theta)=:c_0>0$ and, as a consequence,
	\[
	c_{k-1} + c_0\leq \sum_{i=1}^{k} I_V(\omega_{\Theta_i})\quad\text{for every \ }\{\Theta_1, \ldots, \Theta_k\} \in \mathcal{P}_k, \ \ k \geq 2.
	\]
Therefore, $c_{k-1}<c_{k-1}+ c_0 \leq c_k$. Define
\begin{equation}\label{eq:ell}
\ell_k := c_{\infty}^{-1}c_k.
\end{equation}
Then, $\ell_{k-1}<\ell_k$. Let us show that these numbers have the required property.

	Fix $m \in \mathbb{N}$. Let $G$ be a closed subgroup of $O(N)$ and $\Omega$ be a $G$-invariant domain such that $\o\supset\Theta$ and
    \begin{equation}\label{condition_1}
  \ell_m < \min_{x\in\rn\smallsetminus\{0\}} \# Gx.
    \end{equation}
Then,
    \[
 c_m < \Big( \min_{x\in\rn\smallsetminus\{0\}} \# Gx \Big) c_{\infty}.
    \] 
	Given $\varepsilon \in (0, c_0)$ with $c_m + \varepsilon < \Big(\min\limits_{x\in\rn\smallsetminus\{0\}}\# Gx\Big) c_{\infty} $, we choose $\{\Theta_1, \ldots, \Theta_m\} \in \mathcal{P}_m$ such that
\begin{equation} \label{eq:choice of Thetas}
c_{m} \leq \sum_{i=1}^{m} I_V(\omega_{\Theta_i}) < c_m + \varepsilon.
\end{equation}
Set $\Lambda_m:=\Theta_1\cup\cdots\cup\Theta_m$. Then $\Lambda_m$ is a bounded $G$-invariant open subset of $\o$
and $\omega_{\Theta_i}\in\cN(\Theta_i)^{O(N)}\subset\cN(\Lambda_m)^G$.
For each $k=1,\ldots, m$, we consider the linear subspace of $D^{1,2}_0(\Lambda_m)^G$ given by
	\[
	W_k:= \left\{\sum_{i=1}^k t_i\omega_{\Theta_i}:t_i\in\r\right\}.
	\]
Since $\Theta_i \cap \Theta_j = \emptyset$ if $i \neq j$, we have that $\omega_{\Theta_i}$ and $\omega_{\Theta_j}$ are orthogonal in $D^{1,2}_0(\Lambda_m)^G$. Hence, $\dim (W_k)= k$. Furthermore, using Lemma \ref{lem:nehari}$(vi)$ we see that
$$I_V\Big(\sum_{i=1}^k t_i\omega_{\Theta_i}\Big)=\sum_{i=1}^k I_V(t_i\omega_{\Theta_i})\leq \sum_{i=1}^k I_V(\omega_{\Theta_i})\qquad\text{for all \ }t_1,\ldots,t_k\in\r.$$
Therefore,
$$d_k:= \sup_{W_k} I_V=\sum_{i=1}^k I_V(\omega_{\Theta_i})$$
and, from \eqref{eq:choice of Thetas} and our choice of $\eps$, we get
$$d_k<\Big(\min\limits_{x\in\rn\smallsetminus\{0\}}\# Gx\Big) c_{\infty}\qquad\text{for every \ }k=1,\ldots,m.$$
Then, Theorem \ref{thm:compactness} states that $I_V$ satisfies $(PS)_c$ in $D^{1,2}_0(\Omega)^G$ for all $c \leq d_k$. So, applying Theorem \ref{thm:nodal principle} to each $W_k$ successively, we obtain $m$ pairs of critical points $\pm u_1, \ldots, \pm u_m$ such that $u_1>0$, $u_j\in\cE(\Omega)^G$ if $j=2,\ldots,m$,
\begin{equation*}
I_V(u_1)=c(\o)^G \qquad\text{and}\qquad  I_V(u_k) \leq d_k \qquad \text{for each } k=1, \ldots,m.  
\end{equation*}
Note that, as $I_V(\omega_{\Theta_i})\geq c_0$,
\[
I_V(u_k)+(m-k)c_0\leq d_k + (m-k)c_0\leq  \sum_{i=1}^{m} I_V (  \omega_{\Theta_i} ) < c_m + \varepsilon.
\]
Since we have chosen $\varepsilon < c_0$, we derive
\begin{equation*}
   I_V(u_k) < c_m \quad \text{for each \ } k=1, \ldots, m-1,\qquad\text{and}\qquad I_V(u_m) < c_m+\eps. 
\end{equation*}
In fact, if $m=1$, then $I_V(u_1)=c(\o)^G\leq c(\Theta)^{O(N)}\leq c_1$. Next we prove that, if $m\geq 2$, we can choose $u_m\in\cE(\Omega)^G$ such that
\begin{equation*}
 I_V(u_m) \leq c_m . 
\end{equation*}
Let $\varepsilon_n \in (0,c_0)$ be such that $c_m + \varepsilon_n < \Big(\min\limits_{x\in\rn\smallsetminus\{0\}}\# Gx\Big) c_{\infty} $ and $\varepsilon_n \rightarrow 0$, and let $u_{m,n}\in\cE(\Omega)^G$ be the $m$-th critical point obtained by applying the previous argument with $\varepsilon = \varepsilon_n$. Then, 
$$I_V(u_{m,n}) < c_m + \varepsilon_n< \Big(\min_{x\in\rn\smallsetminus\{0\}}\# Gx\Big) c_{\infty}.$$
If $I_V(u_{m,n}) \leq c_m$ for some $n_0$, we take $u_m:=u_{m,n_0}$. If $I_V(u_{m,n})>c_m$ for all $n$, then $I_V(u_{m,n}) \rightarrow c_m$. By Theorem \ref{thm:compactness}, $I_V$ satisfies $(PS)_{c_m}$ in $D^{1,2}_0(\Omega)^G$. Therefore, as $I'_V(u_{m,n}) =0$, a subsequence of $(u_{m,n})$ converges to $u_m$ in $D^{1,2}_0(\Omega)^G$. Hence, $u_m$ is a critical point of $I_V$ with $I_V(u_m) = c_m$ and, since $\cE(\Omega)^G$ is closed, $u_m\in\cE(\Omega)^G$.

To prove the last statement of Theorem \ref{thm:main2}, recall that $(\ell_k)$ is increasing. Therefore, if \eqref{condition_1} holds true, then
$$\min_{x\in\rn\smallsetminus\{0\}}\# Gx > \ell_k\qquad\text{for every \ }k = 1,\ldots,m,$$
and, as we just saw, for each such $k$ there exist $k$ pairs of $G$-invariant solutions $\pm v_{k,1}, \ldots , \pm v_{k,k}$ to the problem \eqref{eq:problem} such that $v_{k,1}>0$, $v_{k,2}, \ldots v_{k,k}\in\cE(\o)^G$, 
\[
I_V(v_{k,1})=c(\o)^G\qquad\text{and}\qquad I_V(v_{k,i}) \leq c_k \quad \text{for every} \quad i= 1,\ldots,k.
\]
We set $u_1 := v_{1,1}$, $u_2 := v_{2,2}$ and, for each $2<k\leq m$, choose $u_k \in \{v_{k,2},\ldots,v_{k,k}\}$ such that $u_k \neq v_{k-1,i}$ for every $i=1, \ldots, k-1$. Then, $\pm u_1, \ldots, \pm u_m$ are $m$ different pairs of $G$-invariant solutions to problem \eqref{eq:problem} such that $u_1$ is positive, $u_2, \ldots , u_m$ change sign, 
\[
I_V(u_1)=c(\o)^G\qquad\text{and}\qquad I_{V}(u_k) \leq c_k=\ell_kc_\infty \quad \text{for each} \quad k = 1, \ldots, m,
\]
as claimed.
\end{proof}

\bigskip

\begin{flushleft}
\textbf{Mónica Clapp} and \textbf{Carlos Culebro}\\
Instituto de Matemáticas\\
Universidad Nacional Autónoma de México \\
Campus Juriquilla\\
76230 Querétaro, Qro., Mexico\\
\texttt{monica.clapp@im.unam.mx}\\
\texttt{carlos.erick@im.unam.mx} 
\end{flushleft}
	
\end{document}